\documentclass[12pt]{article}
\usepackage{amsmath, amsthm, amssymb,latexsym,mathscinet}
\usepackage{enumerate}
\usepackage{hyperref}
\usepackage{caption,subcaption,xcolor}
\usepackage{tikz,graphicx,cite,todonotes}

\newtheorem{thm}{Theorem}[section]

\newtheorem{cor}[thm]{Corollary}
\newtheorem{clm}[thm]{Claim}

\newtheorem{ques}[thm]{Question}
\newtheorem{example}[thm]{Example}
\numberwithin{equation}{section}

\newcommand{\ir}{i\gamma}

\newcommand{\mes}{\textrm{mes}}
\newcommand{\I}{\mathcal{I}}
\newcommand{\NC}{\mathcal{NC}}

\begin{document}

\title{Collapsibility of non-cover complexes of graphs}

\author{Ilkyoo Choi\footnote{Ilkyoo Choi was supported by the Basic Science Research Program through the National Research Foundation of Korea (NRF) funded by the Ministry of Education (NRF-2018R1D1A1B07043049), and also by the Hankuk University of Foreign Studies Research Fund.}\\[1.5ex]
\small Department of Mathematics\\
\small Hankuk University of Foreign Studies  \\
\small Yongin, Republic of Korea \\
\small\tt ilkyoo@hufs.ac.kr \\[2ex]
\and
Jinha Kim
\\[1.5ex]
\small Department of Mathematics\\
\small Seoul National University \\
\small Seoul, Republic of Korea \\
\small\tt kjh1210@snu.ac.kr \\[2ex]
\and
Boram Park\footnote{Boram Park work  supported by Basic Science Research Program through the National Research Foundation of Korea (NRF) funded by the Ministry of Science, ICT and Future Planning (NRF-2018R1C1B6003577).} \\[1.5ex]
\small Department of Mathematics\\
\small Ajou University\\
\small Suwon, Republic of Korea \\
\small\tt borampark@ajou.ac.kr}

\date\today

\maketitle

\begin{abstract}
Given a graph $G$, the {\em non-cover complex} of $G$ is the combinatorial Alexander dual of the independence complex of $G$.
Aharoni asked if the non-cover complex of a graph $G$ without isolated vertices is $(|V(G)|-\ir(G)-1)$-collapsible where $\ir(G)$ denotes the independent domination number of $G$.
Extending a result by the second author, who verified Aharoni's question in the affirmative for chordal graphs, we prove that the answer to the question is yes for all graphs.
Namely, we show that for a graph $G$, the non-cover complex of a graph $G$ is $(|V(G)|-\ir(G)-1)$-collapsible.
\end{abstract}

\section{Introduction}\label{sec:intro}
We consider only finite simple graphs.
For simplicity, define $[n]:=\{1, \ldots, n\}$. 
Given a graph $G$, let $V(G)$ and $E(G)$ denote the vertex set and edge set, respectively, of $G$. 
An \textit{independent set} of a graph is a subset of the vertices that induces no edge.
Given a graph $G$, a \textit{cover} of $G$ is a subset $W$ of the vertices such that $V(G)\setminus W$ is an independent set of $G$; in other words, $W$ contains an endpoint of every edge of $G$. 
A subset of the vertices that is not a cover is called a \textit{non-cover}.

Given a graph $G$, the {\em independence complex} $\I(G)$ of $G$ is a simplicial complex defined as
\[\I(G) := \{I \subseteq V(G): I\text{ is an independent set of }G\}.\]
The {\em combinatorial Alexander dual} $D(\I(G))$ of $\I(G)$ is defined as
\[D(\I(G)) := \{W \subseteq V(G): V(G) \setminus W \notin \I(G)\}, \]
and is the simplicial complex of non-covers of $G$;
this complex, denoted $\NC(G)$, is also known as the {\em non-cover complex}  of $G$.
In other words, a set $W \subseteq V(G)$ is a member of $\NC(G)$ if and only if $W$ is a non-cover of $G$.
Note that the non-cover complex of a graph with no edges is the void complex.
If a graph with an isolated vertex $v$ has an edge, then the non-cover complex is a cone with apex $v$, and thus it is contractible. 
However, in general, it is not easy to determine the non-cover complex of an arbitrary graph.
Our main result connects the collapsibility of the non-cover complex and the independent domination number of the associated graph. 
We now introduce these two parameters. 

\smallskip

For a graph $G$ and $A,D \subseteq V(G)$, if each $v\in A$ has a neighbor in $D$, then we say $D$ {\em dominates} $A$. 
We use $\gamma(G;A)$ to denote the minimum size of a set that dominates $A$.
The \textit{independent domination number} $\ir(G)$ of $G$ is defined as
\[
\ir(G):=\max \{\gamma(G; I) : I \text{ is an independent set of }G\}.
\]
By convention, we let $\ir(G) = \infty$ when $G$ contains an isolated vertex.

\smallskip

For a finite simplicial complex $X$, a face $\sigma \in X$ is \textit{free} if there is a unique facet of $X$ containing $\sigma$.
An \textit{elementary $d$-collapse} of $X$ is the operation of deleting all faces containing a free face of size at most $d$. 
We say $X$ is $d$-\textit{collapsible} if we can obtain the void complex from $X$ by a finite sequence of elementary $d$-collapses.
The notion of $d$-collapsibility of simplicial complexes was introduced in \cite{Weg75} and has been widely studied ever since \cite{Lew,MT09}. 
An easy observation is that an elementary $d$-collapse does not affect the (non-)vanishing property of homology groups of dimension at least $d$.
See also \cite{Kalai,KM05} for applications regarding Helly-type theorems.
In addition, the topological colorful Helly theorem~\cite{KM05} tells us that given a graph $G$ with a $d$-collapsible non-cover complex, 
for every $d+1$ covers $W_1,\ldots,W_{d+1}$ of $G$, 
there is a cover $W=\{w_{i_1},\dots,w_{i_k}\}$ of $G$ such that $1 \leq i_1 < \cdots < i_k \leq d+1$ and $w_{i_j} \in W_{i_j}$ for each $j \in [k]$; 
the set $W$ is also known as a {\em rainbow cover} of $G$ for $W_1,\dots,W_{d+1}$.

 \smallskip

The collapsibility of non-cover complexes of graphs is related to the {topological connectivity} of independence complexes.
For a simplicial complex $X$, let $\eta(X)$ be the maximum integer $k$ such that $\tilde{H}_j(X)=0$ for all $-1 \leq j \leq k-2$.
(We use $\tilde{H}_i(X)$ to denote the $i$th reduced homology group of $X$ over $\mathbb{Q}$.)
Here, $\tilde{H}_{-1}(X) = 0$ if and only if $X$ is non-empty.
In \cite{Chudnovsky, AH00} (see also \cite{Mes01, Mes03}), it was shown that large independence domination numbers of graphs gives high connectivity of the independence complexes of graphs, in particular, Theorem~\ref{lower-bd}.
Research in this direction was motivated by a topological version of Hall's marriage theorem~\cite{AH00}.

\begin{thm}[\cite{Chudnovsky, AH00}]\label{lower-bd}
For every graph $G$, $\eta(\I(G)) \geq \ir(G)$.
\end{thm}
As a consequence of Theorem~\ref{lower-bd} and the Alexander duality theorem\footnote{\textbf{Alexander duality theorem}(\cite{BBM97}) \
Let $X$ be a simplicial complex on the vertex set $V$.
If $V \notin X$, then for all $-1 \leq i \leq |V|-2$, 
$\tilde{H}_i(D(X)) \cong \tilde{H}_{|V|-i-3}(X)$.} (see \cite{BBM97})
we obtain that for every graph $G$ with at least one edge, the reduced homology group of the non-cover complex of $G$ satisfies
\begin{align}\label{noncover homology}
\tilde{H}_i(\NC(G))=0 \text{ for all } i \geq |V(G)|-\ir(G)-1.
\end{align}

Aharoni~\cite{Aharoni} asked the following question:

\begin{ques}[\cite{Aharoni}]\label{ques-coll}
If $G$ is a graph with no isolated vertices, then is it true that the non-cover complex of $G$ is $(|V(G)|-\ir(G)-1)$-collapsible?
\end{ques}

The verification of Question~\ref{ques-coll} for all graphs implies not only the property in \eqref{noncover homology}, but also the stronger property that for every $W \subseteq V(G)$, the reduced homology group of the subcomplex $\NC(G)[W]$ induced by $W$ satisfies
\[\tilde{H}_i(\NC(G)[W])=0 \text{ for all } i \geq |V(G)| - \ir(G)-1.\]

In \cite{Kim}, the second author of this paper verified Question~\ref{ques-coll} for chordal graphs.
We extend this result by resolving Question~\ref{ques-coll} completely in the affirmative. 

\begin{thm}\label{thm:main:collapse}
For a graph $G$ without isolated vertices, the non-cover complex of $G$ is $(|V(G)|-\ir(G)-1)$-collapsible.
\end{thm}

The main tool for our proof of Theorem~\ref{thm:main:collapse} is minimal exclusion sequences~\cite{MT09} (see also~\cite{Lew}), which we review in section~2 along with the proof of Theorem~\ref{thm:main:collapse}. 
We end the paper by providing some remarks in section~3.

\section{Proof}

\subsection{Minimal exclusion sequences}
In this subsection, we review a result in \cite{MT09}, which will play a key role in the proof.  

For a simplicial complex $X$ on the vertex set $[n]$,
take a linear ordering $\prec: \sigma_1,\ldots,\sigma_m$ of the facets of $X$.
Given a face $\sigma$ of $X$, we define the \textit{minimal exclusion sequence} $\mes_{\prec}(\sigma)$ as follows.
Let $i$ denote the smallest index such that $\sigma \subseteq \sigma_i$.
If $i=1$, then $\mes_{\prec}(\sigma)$ is the null sequence.
If $i\ge 2$, then $\mes_{\prec}(\sigma)=(v_1,\ldots, v_{i-1})$ is a finite sequence of length $i-1$ such that
$v_1=\min (\sigma\setminus \sigma_1)$ and  for each $k\in\{2, \ldots, i-1\}$, 
\[v_k=\begin{cases}
   \min(\{v_1,\dots,v_{k-1}\}\cap (\sigma \setminus \sigma_k)) & \text{if } \{v_1,\dots,v_{k-1}\}\cap (\sigma \setminus \sigma_k)\neq\emptyset,\\
   \min (\sigma\setminus \sigma_k) & \text{otherwise.}
\end{cases} \]
Let $M_{\prec}(\sigma)$ denote the set of vertices appearing in $\mes_{\prec}(\sigma)$, and define
 \[d_{\prec}(X):=\max_{\sigma \in X}|M_{\prec}(\sigma)|.\]
The following was proved in \cite{MT09} (see also \cite{Lew}).
\begin{thm}[\cite{MT09}]\label{thm:minimal_exclusion}
 If $\prec$ is a linear ordering of the facets of $X$, then $X$ is $d_{\prec}(X)$-collapsible.
\end{thm}

\subsection{Proof of Theorem~\ref{thm:main:collapse}}
Let $G$ be a 
graph without isolated vertices.
For simplicity, assume $V(G)=[n]$ and denote $\overline{S}:=[n]\setminus S$ for $S\subseteq [n]$.
Let $I$ be an independent set of $G$ such that $\gamma(G;I)=\ir(G)$.
Let $|I|=i$.
We may assume that $I$ is a maximal independent set and $I:=[n]\setminus [n-i]$.

Note that every facet of $\NC(G)$ is the complement of an edge of $G$.
We define a linear ordering $\prec$ of the facets of $\NC(G)$ as follows.
For two edges $a_1b_1$ and $a_2b_2$, where $a_i< b_i$ for $i\in[2]$,
we denote $a_1b_1 <_{L} a_2b_2$ if either (i) $b_1<b_2$ or (ii) $b_1=b_2$ and $a_1<a_2$.
For two distinct facets $\sigma$ and $\tau$ of $\NC(G)$, we denote $\sigma \prec \tau$ if $\overline{\sigma}<_{L}\overline{\tau}$. 

\begin{clm}\label{lem:mes}
For $\sigma, \sigma' \in \NC(G)$, if $\overline{\sigma}\cap \overline{I}=\overline{\sigma'}\cap \overline{I}$ and $G[\overline{\sigma}\cap \overline{I}]$ contains an edge,
then $\mes_{\prec}(\sigma)=\mes_{\prec}({\sigma'})$.
\end{clm}

\begin{proof}
Let $j$ be the length of $\mes_{\prec}(\sigma)$.
Note that an edge between $I$ and $\overline{I}$ comes before all the edges of $G[\overline{I}]$ in the linear ordering $<_L$.
Since $G[\overline{\sigma}\cap \overline{I}]$ has an edge, 
for the $(j+1)$th facet $\sigma_{j+1}$, $\overline{\sigma_{j+1}}$ is an edge such that $\overline{\sigma_{j+1}}\subseteq \overline{I}$.
By the definition of $\prec$, it also follows that 
for every $k \in [j+1]$, the $k$th facet $\sigma_k$ satisfies $\overline{\sigma_k}\subseteq \overline{I}$.  
Clearly, $ \sigma \cap \overline{I}= \sigma' \cap \overline{I}$.
Thus, we have
\[\overline{\sigma_k} \cap \sigma=\overline{\sigma_k} \cap \sigma \cap \overline{I}=\overline{\sigma_k} \cap \sigma'\cap \overline{I}=\overline{\sigma_k} \cap \sigma'.\]
Thus the length of $\mes_{\prec}(\sigma')$ is also $j$ and for every $k \in [j]$,
the $k$th entry of $\mes_{\prec}(\sigma)$ is equal to that of $\mes_{\prec}(\sigma')$.
\end{proof}

\begin{clm}\label{clm}
For every $S\subseteq \overline{I}$,
\[
|S|-|N(S)\cap I| \geq \ir(G)-|I|,
\]
where $N(S)=\{v \in V(G) : uv \in E(G) \text{ for some }u \in S\}$.
\end{clm}
\begin{proof}
We take $S \subseteq \overline{I}$ so that (1) $|S|-|N(S)\cap I|$ is minimum, and (2) $|S|$ is maximum subject to (1).
By the minimality of $|S|-|N(S)\cap I|$,
every element in $\overline{I}\setminus S$ has at most one neighbor in $I\setminus N(S)$.
If some $v\in \overline{I}\setminus S$ has exactly one neighbor $w$ in $I\setminus N(S)$, then for $T=S\cup\{w\}\subseteq \overline{I}$, we know $|T|-|N(T)\cap I|=|S|-|N(S)\cap I|$ and $|T|>|S|$, which is a contradiction to the maximality of $|S|$.
Thus, every element in $\overline{I}\setminus S$ does not have a neighbor in $I\setminus N(S)$.
Since $G$ has no isolated vertex,  we conclude $N(S) \cap I=I$.
Hence, $S$ dominates $I$ and so $|S| \geq \ir(G)$.
Thus $
|S|-|N(S)\cap I| \geq \ir(G)-|I|$.
\end{proof}

By Theorem~\ref{thm:minimal_exclusion}, it is sufficient to show that 
\begin{eqnarray}\label{eq:M}
&&|M_{\prec}(\sigma)|\le |V(G)|-\ir(G)-1\quad \text{ for every  }\sigma \in \NC(G).
\end{eqnarray}

For a face $\sigma \in \NC(G)$, let $\beta(\sigma)=|N(\overline{\sigma}\cap \overline{I})\cap \overline{\sigma} \cap I|$.
Suppose that $\beta(\sigma)=0$. 
Then $G[\overline{\sigma} \cap \overline{I}]$ must have an edge.
Consider $\sigma'=\sigma\cup I$.
Then $\overline{\sigma}\cap \overline{I}=\overline{\sigma'}\cap \overline{I}$.
By Claim~\ref{lem:mes}, $\mes_{\prec}(\sigma)=\mes_{\prec}(\sigma')$ and therefore,
$M_{\prec}(\sigma)=M_{\prec}(\sigma')$.
On the other hand, we know $\beta(\sigma')\ge 1$ by the definition of $\sigma'$.
Thus, it is sufficient check \eqref{eq:M} under the assumption $\beta(\sigma)\ge 1$.

Note that for $v \in \sigma \cap I$, if $v \in M_{\prec}(\sigma)$, then $v$ is a neighbor of some vertex in $\overline{\sigma} \cap \overline{I}$.
Thus,
\begin{eqnarray*}\label{eq:mes}
|M_{\prec}(\sigma)| &\le& |\sigma \cap \overline{I}|  + |N(\overline{\sigma}\cap \overline{I})\cap (\sigma \cap I)|\\
&=& |\overline{I}|-|\overline{\sigma}\cap \overline{I}| +
|N(\overline{\sigma}\cap \overline{I})\cap I|-\beta(\sigma)\\ 
 &\le &|\overline{I}|-\ir(G)+|I|-\beta(\sigma)\\
 &=&|V(G)|-\ir(G)-\beta(\sigma),
\end{eqnarray*}
where the last inequality holds by applying Claim~\ref{clm} to the set $\overline{\sigma}\cap \overline{I}$.
As we assumed that $\beta(\sigma)\ge 1$, \eqref{eq:M} follows, and this concludes the proof of Theorem~\ref{thm:main:collapse}.

\section{Concluding remarks}

For a graph $G$ and $A,W \subseteq V(G)$, if each $w\in A$ has a neighbor in $W$ or $w\in W$, then we say $W$ {\em weakly dominates} $A$.
We use $\gamma_w(G;A)$ to denote the minimum size of a set that weakly dominates $A$.
The \textit{weak independent domination number} $\ir_w(G)$ of $G$ is defined as
\[
\ir_w(G):=\max \{\gamma_w(G; I) : I \text{ is an independent set of }G\}.
\]
The following is a straightforward application of Theorem~\ref{thm:main:collapse}.

\begin{cor}
For a graph $G$, the non-cover complex of $G$ is $(|V(G)|-\ir_w(G)-1)$-collapsible.
\end{cor}

\begin{proof}
If $G$ has no isolated vertex, then $\ir_w(G)=\ir(G)$ and we are done by Theorem~\ref{thm:main:collapse}.
Assume $G$ has $k$ isolated vertices for some integer $k \geq 1$. 
Let $W$ be the set of isolated vertices of $G$, and let $G'$ be the graph obtained from $G$ by removing all vertices in $W$.

Recall that $\NC(G)$ is a cone with apex $v$ if $v$ is an isolated vertex of $G$.
Thus $\NC(G)$ is $d$-collapsible if and only if the subcomplex of $\NC(G)$ induced by $V(G) \setminus \{v\}$ is $d$-collapsible.
Moreover, since the subcomplex of $\NC(G)$ induced by $V(G) \setminus W$ is equal to $\NC(G')$, it follows that $\NC(G)$ is $d$-collapsible if and only if $\NC(G')$ is $d$-collapsible.
Thus, it is sufficient to show $\NC(G')$ is $(|V(G)|-\ir_w(G)-1)$-collapsible.
By Theorem~\ref{thm:main:collapse}, $\NC(G')$ is $(|V(G')|-\ir(G')-1)$-collapsible.
Since $|V(G')|=|V(G)|-k$ and $\ir_w(G)=\ir(G')+k$, we obtain $|V(G')|-\ir(G')-1=|V(G)|-\ir_w(G)-1$.
\end{proof}

We finish the section by stating a direct consequence of the topological colorful Helly theorem~\cite{KM05} from our main result.

\begin{cor}\label{r-cover-coll}
Let $G$ be a graph on $n$ vertices and let $W_1,\dots,W_{n-\ir(G)} \subseteq V(G)$.
Assume that every set $A \subseteq V(G)$ satisfying the following two conditions is a cover of $G$:
\begin{enumerate}[(i)]
\item $A \cap W_i \ne \emptyset$ for $i \in [n-\ir(G)]$.
\item $W_j \subseteq A$ for some $j \in [n-\ir(G)]$.
\end{enumerate}
Then there is a cover $W$ of $G$ where $W=\{w_{i_1},\dots,w_{i_k}\}$ with $1 \leq i_1 <\cdots <i_k \leq n-\ir(G)$ and $w_{i_j} \in W_{i_j}$ for each $j \in [k]$.
\end{cor}

Dao and Schweig \cite{DS13} showed a weaker version of Theorem~\ref{thm:main:collapse} concerning a topological property known as ``Lerayness" via an algebraic approach.
Let us briefly introduce their result. 
For a simplicial complex $X$, we say $X$ is {\it $d$-Leray} if $\tilde{H}_i(Y)=0$ for all induced subcomplexes $Y$ of $X$ and all integers $i \geq d$.
Wegner showed that $d$-collapsiblity implies $d$-Lerayness\cite{Weg75}, yet the converse is not always true\cite{MT09}.
Hochster \cite{Hochster} proved the relation between the Leray number\footnote{For a simplicial complex $X$, the {\it Leray number} of $X$ is the minimum integer $k$ such that $X$ is $k$-Leray.} and the Castelnuovo-Mumford regularity of the Stanley-Reisner ideal of a simplicial complex.
From this relationship and the result in \cite{DS13}, it was shown that for a graph $G$, the non-cover complex $\NC(G)$ is $(|V(G)|-\ir(G)-1)$-Leray.
There is an active line of research in this direction, see \cite{KM06, Woodroofe} for more details.  
By applying the topological colorful Helly theorem of the Lerayness version, we obtain the following:

\begin{cor}\label{r-cover-leray}
Let $G$ be a graph on $n$ vertices.
For every $n-\ir(G)$ covers $W_1,\dots,W_{n-\ir(G)}$ of $G$, 
there is a cover $W$ of $G$ where $W=\{w_{i_1},\dots,w_{i_k}\}$ with $1 \leq i_1 <\cdots <i_k \leq n-\ir(G)$ and $w_{i_j} \in W_{i_j}$ for each $j \in [k]$.
\end{cor}

Note that Corollary~\ref{r-cover-leray} is weaker than Corollary~\ref{r-cover-coll}, since if we have $n-\ir(G)$ covers for a graph $G$, then a set $A \subseteq V(G)$ satisfying (ii) is a cover of $G$.
As mentioned in the introduction, the set $W$ in Corollary~\ref{r-cover-coll} and~\ref{r-cover-leray} is also known as a {\em rainbow cover} of $G$ for $W_1,\dots,W_{n-\ir(G)}$. 
The following example demonstrates that Corollaries~\ref{r-cover-coll} and~\ref{r-cover-leray} are tight.

\begin{example}\rm
Let $C_{3k}$ be a cycle of length $3k$ for an integer $k \geq 2$.
It is easy to verify $\ir(C_{3k})=k$ and so $|V(C_{3k})|-\ir(C_{3k})=2k$.
Consider $M \subseteq V(C_{3k})$ that induces a matching of size $k$, so that $M$ is a cover of $C_{3k}$.
Let $W_i=M$ for all $i \in [2k-1]$.
It is again easy to verify that there is no rainbow cover with respect to $W_1,\dots,W_{2k-1}$.
\end{example}

\section*{Acknowledgements}
The authors thank professor Ron Aharoni for introducing the problem to the second author.
This work was done during the 4th Korean Early Career Researcher Workshop in Combinatorics.

\bibliographystyle{plain}

\end{document}